\documentclass[reqno,11pt]{amsart}

\usepackage{xcolor}
\usepackage{graphicx}
\usepackage[hidelinks]{hyperref}
\usepackage{amscd}

\usepackage{amsmath}
\usepackage{amsthm}
\usepackage{amssymb}
\usepackage{latexsym,array}
\usepackage{amsfonts}
\usepackage{shadow}
\usepackage{amsbsy}
\usepackage{amssymb}
\usepackage{dsfont}
\usepackage{mathtools}
\usepackage{enumitem}
\usepackage{doi}
\usepackage{color}
\usepackage{graphicx}
\usepackage[hidelinks]{hyperref}
\usepackage{cleveref}
\usepackage{fullpage}
\usepackage{comment}

\usepackage{dsfont}

\usepackage{cleveref}

\numberwithin{equation}{section}

\newtheorem{theorem}{Theorem}[section]
\newtheorem{lemma}[theorem]{Lemma}

\newtheorem{proposition}[theorem]{Proposition}

\theoremstyle{definition}
\newtheorem{remark}[theorem]{{\bf Remark}}
\newtheorem{definition}[theorem]{Definition}

\newcommand{\unx}{\underline{x}}

\newcommand{\ku}{k_1}
\newcommand{\kd}{k_2}

\newcommand{\qcsa}[1]{\mathcal Q^{#1}_{c,s}(x)}
\usepackage{amscd}
\usepackage{tikz-cd}
\usepackage{tikz}

\crefname{enumi}{}{}
\crefname{enumii}{}{}

\title[]{On the application of the factorized Fueter-Sce map to the slice hyperholomorphic Cauchy kernel}

\author[A. De Martino]{Antonino De Martino}
\address{(ADM)
	Politecnico di Milano\\Dipartimento di Matematica\\Via E. Bonardi, 9\\20133
	Milano, Italy
} \email{antonino.demartino@polimi.it}

\author[S. Pinton]{Stefano Pinton}
\address{(SP)
	Politecnico di Milano\\Dipartimento di Matematica\\Via E. Bonardi, 9\\20133
	Milano, Italy
} \email{stefano.pinton@polimi.it}

\date{}

\begin{document}
	
\maketitle

\begin{abstract}
The Fueter-Sce theorem is one of the most important results in hypercomplex analysis, providing a two-step procedure for constructing axially monogenic functions starting from holomorphic functions of one variable. In the first step, the so-called slice operator is applied to holomorphic functions of one variable, producing the class of slice hyperholomorphic functions. The second step yields the class of axially monogenic functions by applying the pointwise differential operator $\Delta_{n+1}^{\frac{n-1}{2}}$, with $n$ odd, known as the Fueter-Sce map. The significance of the Fueter–Sce theorem also lies in the fact that it induces two spectral theories corresponding to the function classes it generates.

Over the years, factorizations of the Fueter-Sce map have been studied to identify the intermediate spaces that arise between slice hyperholomorphic functions and axially monogenic functions. Until now, only certain factorizations of the Fueter–Sce map have been considered. In this paper, our goal is to determine the most general factorizations of the Fueter–Sce map and to apply these differential operators to the Cauchy kernel of slice hyperholomorphic functions.
\end{abstract}

\noindent \textbf{AMS Classification:} \\


\section{Introduction}

In the literature, the notion of holomorphic functions of a one complex variable can be generalized to higher dimensions in two distinct ways. One method relies on studying systems of Cauchy–Riemann equations for functions of several complex variables, which gives rise to the theory of several complex variables.

An alternative approach is to examine functions taking values in quaternions or in Clifford algebras. Within this framework, two central theories have emerged: slice hyperholomorphic function theory (\cite{CSS2016}) and monogenic function theory (\cite{red, green}).These two theories of hypercomplex analysis differ significantly. For instance, polynomials are slice hyperholomorphic, but they are not monogenic. Despite this fundamental distinction, a connection between slice hyperholomorphic function theory and monogenic function theory is established through the Fueter–Sce extension theorem, see \cite{ColSabStrupSce, Fueter, Sce}.

This construction consists of two steps. In the first step, the space of holomorphic functions of one complex variable, denoted by $\mathcal{O}(\Pi)$, is extended via the slice operator $T_{FS1}$ to the space of slice hyperholomorphic functions $\mathcal{SH}(U_\Pi)$, where $\Pi \subset \mathbb{C}$ is an open symmetric set and $U_\Pi \subset \mathbb{R}^{n+1}$ is the corresponding induced domain. 

In the second step, the Fueter-Sce operator $\Delta_{n+1}^{\frac{n-1}{2}}$, defined for odd $n$, is applied to $\mathcal{SH}(U_\Pi)$, yielding the space of axially monogenic functions $\mathcal{AM}(U_\Pi)$. The entire procedure can be summarized as
\[
\begin{CD}
	\mathcal{O}(\Pi) @>T_{FS1}>> \mathcal{SH}(U_\Pi) @>\Delta_{n+1}^{\frac{n-1}{2}}>> \mathcal{AM}(U_\Pi).
\end{CD}
\]

The Fueter-Sce theorem has the important consequence of generating two spectral theories, respectively in the settings of slice hyperholomorphic and axially monogenic functions.

Over the past years, the authors have investigated factorizations of the Fueter-Sce map in terms of Dirac operators. The study began in \cite{CDPS, Polyf1, Polyf2} with the quaternionic case ($n=3$). In \cite{Fivedim}, the factorizations of $\Delta_{n+1}^{\frac{n-1}{2}}$ were analyzed for $n=5$. More recently, in \cite{CDP25}, the focus was placed on the main factorizations of the Fueter-Sce map. The various families of functions, together with the corresponding functional calculi induced by these factorizations, give rise to the notion of \emph{fine structure}; see \cite{S1, S2A} for surveys.

The goal of this paper is to determine the most general factorizations of the operator $\Delta_{n+1}^{\frac{n-1}{2}}$ in terms of the Dirac operator and its conjugate, and to apply the resulting operators from these factorizations to the slice hyperholomorphic Cauchy kernels.
\newline
\newline
\emph{Outline of the paper:} Section 1 is this introduction. In Section 2, we review the main notions of slice hyperholomorphic analysis as well as the monogenic case. Section 3 begins with a summary of the factorizations of the Fueter-Sce map studied in previous works and their applications to the slice hyperholomorphic Cauchy kernel. In the second part of this section, we apply the most general factorizations of the Fueter--Sce map to the slice hyperholomorphic Cauchy kernel, and we observe that, for particular choices of the parameters, these applications yield kernels of integral transforms previously obtained in other works. Finally, Section 4 provides some concluding remarks.

\section{Preliminaries}

In this paper, the framework is the real Clifford algebra $\mathbb{R}_n$, generated by $n$ imaginary units $e_1, \ldots, e_n$ satisfying the relations $e_i e_j + e_j e_i = -2 \delta_{ij}$.  
An element of the Clifford algebra $\mathbb{R}_n$ has the form $\sum_{A} e_A x_A$, where $A = (i_1, \ldots, i_r) \subseteq \{1,2,\ldots,n\}$ with $i_1 < \cdots < i_r$ is a multi-index, $e_A = e_{i_1} e_{i_2} \cdots e_{i_r}$, and $e_{\emptyset} = 1$.  

If $|A| = i_1 + \cdots + i_r$, then elements of the form $\sum_{A \,:\, |A|=k} e_A x_A$, with $k>0$, are called $k$-vectors.  
Note that $\mathbb{R}_1$ coincides with the algebra of complex numbers $\mathbb{C}$, while $\mathbb{R}_2$ corresponds to the real quaternion algebra, denoted by $\mathbb{H}$. For $n>1$, the Clifford algebras $\mathbb{R}_n$ are noncommutative, and for $n>2$ they also contain zero divisors (see \cite{GHS}).
\\An element $(x_1, \ldots, x_n) \in \mathbb{R}^n$ can be identified with a $1$-vector in the Clifford algebra via the map
\[
(x_1, \ldots, x_n) \longmapsto \underline{x} = x_1 e_1 + \cdots + x_n e_n,
\]
while an element $(x_0, x_1, \ldots, x_n) \in \mathbb{R}^{n+1}$ can be identified with a paravector of the form
\[
x = x_0 + \underline{x} = x_0 + \sum_{j=1}^n x_j e_j.
\]
The real part $x_0$ of $x$ will be also denoted by $\hbox{Re}(x)$.  The norm of $x \in \mathbb{R}^{n+1}$ is given by $|x|^2=x_0^2+....x_n^2$. The symbol $\mathbb{S}$ denotes the $(n-1)$-dimensional sphere of unit $1$-vectors in $\mathbb{R}^n$, i.e.
$$
\mathbb{S} = \left\{ \underline{x} = x_1 e_1 + \cdots + x_n e_n \, : \, x_1^2 + \cdots + x_n^2 = 1 \right\}.
$$
We observe that any element $I \in \mathbb{S}$ satisfies $I^2 = -1$. Thus, one can consider the real vector space
\[
\mathbb{C}_I = \mathbb{R} + I\mathbb{R} = \{ u + Iv \, : \, u,v \in \mathbb{R} \}.
\]
For each $I \in \mathbb{S}$, the set $\mathbb{C}_I$ is a $2$-dimensional real subspace of $\mathbb{R}^{n+1}$, isomorphic to the complex plane. Given an element $x = x_0 + \underline{x} \in \mathbb{R}^{n+1}$, we define $ J_x = \frac{\underline{x}}{|\underline{x}|}$, if  $\underline{x} \neq 0$. Then, for such an element $x$, we introduce the set
$$
[x] := \{ y \in \mathbb{R}^{n+1} \, : \, y = x_0 + J |\underline{x}|, \; J \in \mathbb{S} \},
$$
which is an $(n-1)$-dimensional sphere in $\mathbb{R}^{n+1}$.

\begin{definition}
Let $U \subseteq\mathbb{R}^{n+1}$.We say that $U$ is axially symmetric if $[x] \in U$ for every $x \in U$.
\end{definition}

\begin{definition}[Slice Cauchy domain]
An axially symmetric open set $U \subset \mathbb{R}^{n+1}$ is called a \emph{slice Cauchy domain} if $U \cap \mathbb{C}_I$ is a Cauchy domain in $\mathbb{C}_I$ for every $I \in \mathbb{S}$. More precisely, $U$ is a slice Cauchy domain if, for each $I \in \mathbb{S}$, the boundary $\partial(U \cap \mathbb{C}_I)$ of $U \cap \mathbb{C}_I$ is the union of a finite number of pairwise disjoint, piecewise continuously differentiable Jordan curves in $\mathbb{C}_I$.
\end{definition}

\begin{definition}[Slice hyperholomoprhic functions]
Let $U \subseteq \mathbb{R}^{n+1}$ be an axially symmetric set and let
$$ \mathcal{U}:= \{(u,v) \in \mathbb{R}^2 \, :\, u+\mathbb{S}v \subset U\}.$$
A function $f: U \to \mathbb{R}^{n+1}$ is called a left (resp. right) slice function, if it is of the form
\begin{equation}
\label{slice}
f(x)=\alpha(u,v)+I \beta(u,v) \quad \left( \, \hbox{resp.} \, \,f(x)=\alpha(u,v)+ \beta(u,v)I\right), \quad \hbox{for} \quad x=u+Iv \in U,
\end{equation}
where the functions $\alpha$, $\beta : \mathcal{U}\to \mathbb{R}_n$ satisfy the so-called even-odd conditions:
\begin{equation}
\label{eo}
\alpha(u,v)=\alpha(u,-v), \qquad \beta(u,v)=-\beta(u,-v), \qquad \forall (u,v) \in \mathcal{U}.
\end{equation}
If in addition $\alpha$ and $\beta$ satify the Cauchy-Riemann equations
$$ \partial_u \alpha(u,v)-\partial_v\beta(u,v)=0, \qquad \partial_v \alpha(u,v)+\partial_u \beta(u,v)=0,$$
then $f$ is called left (resp. right) slice hyperholomorphic.
\end{definition}

\begin{definition}
	Let $U \subset \mathbb{R}^{n+1}$ be an axially symmetric open set.  
	\begin{itemize}
		\item We denote by $\mathcal{SH}_L(U)$ (resp. $\mathcal{SH}_R(U)$) the set of left (resp. right) slice hyperholomorphic functions on $U$. When no confusion arises, we simply write $\mathcal{SH}(U)$.
		\item A slice hyperholomorphic function of the form \eqref{slice} such that $\alpha$ and $\beta$ are real-valued functions is called an intrinsic slice hyperholomorphic function, and the set of such functions will be denoted by $\mathcal{N}(U)$.
	\end{itemize}
\end{definition}

\begin{remark}
The set of intrinsic slice hyperholomorphic functions coincides with the set of left and right slice hyperholomorphic functions. Therefore, when referring to $\mathcal{N}(U)$, we do not distinguish between the left and right cases.
\end{remark}

A well-known example of a slice hyperholomorphic function is the slice hyperholomorphic Cauchy kernel, which plays a fundamental role in this paper, see \cite{CGK, ColomboSabadiniStruppa2011}.

\begin{proposition}[Cauchy kernel series]
Let $s$, $x \in \mathbb{R}^{n+1}$ such that $|x|<|s|$ then
$$ \sum_{n=0}^{\infty} x^n s^{-1-n}=-(x^2-2xs_0+|s|^2)^{-1}(x-\bar{s})$$
and
$$ \sum_{n=0}^{\infty}s^{-1-n} x^n =-(x-\bar{s})(x^2-2xs_0+|s|^2)^{-1}.$$
\end{proposition}

The sum of the Cauchy kernel series can be written in two equivalent ways.

\begin{proposition}
Suppose that $x$ and $s \in \mathbb{R}^{n+1}$ are such that $x \notin [s]$. We have that
$$-(x^2-2xs_0+|s|^2)^{-1}(x-\bar{s})=(s-\bar{x})(s^2-2x_0s+|x|^2)^{-1}$$
and
$$(s^2-2x_0s+|x|^2)^{-1}(s-\bar{x})=-(x-\bar{s})(x^2-2xs_0+|s|^2)^{-1}.$$
\end{proposition}

 These fact justify the following definition

\begin{definition}
Let $x$,$s \in \mathbb{R}^{n+1}$ such that $x \notin [x]$. 
\begin{itemize}
\item We say that $S^{-1}_L(s,x)$ (resp. $S^{-1}_R(s,x)$) is written in the form $I$ if
$$
S^{-1}_L(s,x)=-(x^2-2xs_0+|s|^2)^{-1}(x-\bar{s}), \qquad \left(\hbox{resp.} \quad S^{-1}_R(s,x)=-(x-\bar{s})(x^2-2xs_0+|s|^2)^{-1} \right)
$$
\item We say that $S^{-1}_L(s,x)$ (resp. $S^{-1}_R(s,x)$) is written in the form $II$ if
$$
	S^{-1}_L(s,x)=(s-\bar{x})(s^{2}-2x_0s+|x|^2)^{-1}, \qquad \left(\hbox{resp.} \quad S^{-1}_R(s,x)=(s^{2}-2x_0s+|x|^2)^{-1}(s-\bar{x}) \right).
$$
\end{itemize}
\end{definition}

\begin{remark}
In this paper, the following polynomial plays a crucial role:  
$$
\mathcal{Q}_{c,s}(x) = s^2 - 2x_0 s + |x|^2, \qquad s,x \in \mathbb{R}^{n+1},
$$
together with its inverse, defined for $s, x \in \mathbb{R}^{n+1}$ with $x \notin [s]$,  
\begin{equation}
\label{inveQ}
	\mathcal{Q}_{c,s}^{-1}(x) = \big(s^2 - 2x_0 s + |x|^2\big)^{-1}.
\end{equation}
The function $\mathcal{Q}_{c,s}^{-1}(x)$, occurring in the second representation of the Cauchy kernel, is commonly referred to as the commutative pseudo-Cauchy kernel.
 \end{remark}

\begin{theorem}[The Cauchy formulas for slice hyperholomorphic functions, see \cite{CGK,ColomboSabadiniStruppa2011}]
\label{Cauinte}
Let $U \subset \mathbb{R}^{n+1}$ be a bounded slice Cauchy domain let $I \in \mathbb{S}$ and set $ds_I=ds(-I)$. If $f$ is a left (resp.) slice hyperholomorphic function on a set that contains $\overline{U}$ then for any $x \in U$ we have
$$ f(x)=\frac{1}{2 \pi} \int_{\partial(U \cap \mathbb{C}_I)} S^{-1}_L(s,x) ds_I f(s), \qquad \left( \hbox{resp.} \, \, f(x)=\frac{1}{2 \pi} \int_{\partial(U \cap \mathbb{C}_I)} f(s) ds_I S^{-1}_R(s,x)\right).$$
\end{theorem}

In the context of hypercomplex analysis, we now introduce another relevant class of functions:

\begin{definition}
Let $U \subseteq \mathbb{R}^{n+1}$ be an axially symmetric domain. A function $f: U \to \mathbb{R}^{n+1}$ of class $\mathcal{C}^1$ is a left (resp right) axially monogenic function if it is in the kernel of the Dirac operator $D$:
$$ Df(x)= \left( \frac{\partial}{\partial x_0}+ \sum_{i=1}^{n}e_i \frac{\partial}{\partial x_i}\right)f(x)=0 \qquad \left( \hbox{resp.} \, \, f(x)D=  \frac{\partial}{\partial x_0}f(x)+ \sum_{i=1}^{n} \frac{\partial}{\partial x_i}f(x)e_i=0 \right),$$
and if in addition it has axial form, i.e. it is a left (resp. right) function
$$ f(x)=A(u,v)+\underline{\omega}B(u,v), \qquad \left( \, \, \hbox{resp.} \, \, f(x)=A(u,v)+B(u,v)\underline{\omega} \right), \quad \underline{\omega} \in \mathbb{S},$$
where $A$ and $B$ are functions with values in $\mathbb{R}^{n+1}$, that satisfy the even-odd condition \eqref{eo}. We denote this class of functions by $ \mathcal{AM}_L(U)$ (resp. $\mathcal{AM}_R(U)$). When no confusion arise we denote this class of functions by $\mathcal{AM}(U)$.
\end{definition}

\begin{remark}
The conjugate of the Dirac operator is given by
$$ \overline{D}=\frac{\partial}{\partial x_0}- \sum_{i=1}^{n} e_i \frac{\partial}{\partial x_i}.$$
\end{remark}

A connection between the axially monogenic and slice hyperhlomoprhic functions is given by the Fueter-Sce theorem, see \cite{ColSabStrupSce, Sce}.

\begin{theorem}[Fueter-Sce mapping theorem]
\label{FS}
Let $n$ be an odd number and set $h_n=\frac{n-1}{2}$. We assume that  
$$
f_0(z)= \alpha(u,v)+i \beta(u,v)
$$  
is a holomorphic function of one complex variable defined in a domain $\Pi$ contained in the upper half-plane. Consider  
$$
U_{\Pi}:= \{x=x_0+ \underline{x} \,: \, (x_0,| \underline{x}|) \in \Pi \},
$$ 
which is the open set induced by $\Pi$ in $\mathbb{R}^{n+1}$. The operator $T_{FS1}$, called the \emph{slice operator}, is defined by  
$$
f(x)=T_{FS1}(f_0):=\alpha(x_0, |\underline{x}|)+ \frac{\underline{x}}{| \underline{x}|} \beta(x_0, |\underline{x}|), \qquad x \in U_{\Pi},
$$
and maps holomorphic functions into slice hyperholomorphic functions. Moreover, the operator  
$$
T_{FS2}:= \Delta_{n+1}^{h_n}
$$  
maps slice hyperholomorphic functions into axially monogenic functions. In fact, the function  
$$
\breve{f}(x):= \Delta_{n+1}^{h_n} \left(\alpha(x_0, |\underline{x}|)+ \frac{\underline{x}}{| \underline{x}|} \beta(x_0, |\underline{x}|)\right), \qquad x \in U_\Pi,
$$
belongs to the kernel of the Dirac operator, i.e. $
D \breve{f}(x)=0.
$
\end{theorem}

\begin{remark}
The above theorem was first established for quaternions in 1934, see \cite{Fueter}. In the late 1950s, it was extended to Clifford algebras in the case of odd dimensions, see \cite{Sce} (for an English translation, see \cite{ColSabStrupSce}). In 1997, the Fueter-Sce theorem was further generalized to arbitrary dimensions. In this generalization, the operator $\Delta_{n+1}^{h_n}$ is understood as a fractional operator defined via the Fourier transform, see \cite{TaoQian1, TAOBOOK}. Moreover, in \cite{DG, DDG, DDG1} one can find a connection between the Fueter--Sce theorem and the Cauchy--Kovalevskaya extension.

\end{remark}

\begin{remark}
In Theorem \ref{FS}, one of the hypotheses is that the holomorphic functions are defined only on the upper-half complex plane. However, this condition can be relaxed. This can be achieved by considering functions of the form
$$
f(x)= \alpha(u,v)+I \beta(u,v), \qquad \text{for } x=u+Iv \in U,
$$
where the functions $\alpha, \beta : \mathcal{U} \to \mathbb{R}^{n+1}$ satisfy the even-odd condition \eqref{eo}.

\end{remark}

The extension procedure (given by the Fueter--Sce mapping theorem) from holomorphic functions of cone complex variable to axially monogenic functions (passing through slice hyperholomorphic functions) can be illustrated as follows:

\begin{equation}
\label{SC}
	\begin{CD}
		\textcolor{black}{\mathcal{O}(\Pi)} @>T_{FS1}>> \textcolor{black}{\mathcal{SH}(U_\Pi)} @> T_{F2} = \Delta_{4} >> \textcolor{black}{\mathcal{AM}(U_\Pi)}.
	\end{CD}
\end{equation}

In \cite{CSS}, the authors applied the operator $T_{FS2}:= \Delta_{n+1}^{h_n}$ to the slice hyperholomorphic Cauchy kernel, which leads to the following result.

\begin{theorem}
Let $n$ be an odd number set $h_n= \frac{n-1}{2}$ and consider $x$, $s \in \mathbb{R}^{n+1}$ be such that $s \notin [x]$ and recall that $\mathcal{Q}_{c,s}^{-1}(x)$ is defined in \eqref{inveQ}. Then the application of the operator $T_{FS2}:= \Delta_{n+1}^{h_n}$ to the left (resp. right) slice hyperholomorphic Cauchy kernel in form $II$ we get
$$ \Delta_{n+1}^{h_n}S^{-1}_L(s,x)=F_L^n(s,x)=\gamma_n (s-\bar{x})\mathcal{Q}_{c,s}^{-h_n-1}(x),$$ 
$$ \left(\hbox{resp.} \, \, \Delta_{n+1}^{h_n}S^{-1}_R(s,x)=F_R^n(s,x)=\gamma_n \mathcal{Q}_{c,s}^{-h_n-1}(x) (s-\bar{x})\right),$$
where 
$ \gamma_n:= 4^{h_n} h_n! (-h_n)_{h_n}$.
\end{theorem}

By combination the above theorem and the Cauchy theorem for slice hyperholomorphic functions one can can get an integral version of th Fueter-Sce mapping theorem, see \cite{CSS}.

\begin{theorem}[Fueter-Sce mapping theorem in integral form]
\label{FSinte}
Let $n$ be an odd number and $h_n=\frac{n-1}{2}$. Let $U \subset \mathbb{R}^{n+1}$ be a bounded slice Cauchy domain, let $I \in \mathbb{S}$ and set $ds_I=- I ds$. If $f$ is a left (resp. right) slice hyperholomorphic function on a set that contains $\bar{U}$, then $\breve{f}(x)= \Delta_{n+1}^{h_n} f(x)$ is left (resp. right) axially monogenic and it admits, for any $x \in U$, the following integral representation 
$$ \breve{f}(x)=\frac{1}{2 \pi} \int_{\partial(U \cap \mathbb{C}_I)} F_L^n(s,x)ds_I f(s) \quad \left( \, \, \hbox{resp.} \, \, \breve{f}(x)=\frac{1}{2 \pi} \int_{\partial(U \cap \mathbb{C}_I)}f(s)ds_I F_R^n(s,x)\right).$$

\end{theorem}

\begin{remark}
For $\alpha$, $ \beta >0$ the symbol $(-\alpha)_{\beta}$ have the following interpretation:

$$(-\alpha)_\beta=(-1)^\beta (\alpha-\beta+1)_\beta=(-1)^\beta \frac{\Gamma(\alpha+1)}{\Gamma(\alpha-\beta+1)}.$$

\end{remark}

One of the most significant applications of the Fueter-Sce mapping theorem lies in its capacity to induce two distinct spectral theories corresponding to the two classes of hyperholomorphic functions it generates. By applying the Cauchy formula for slice hyperholomorphic functions (see Theorem \ref{Cauinte}), one obtains the $S$-functional calculus, which is constructed on the basis of the $S$-spectrum, see \cite{ColomboSabadiniStruppa2011}.
The spectral theory for quaternionic operators based on the $S$-spectrum, originally inspired by the work of Birkhoff and von Neumann on quaternionic quantum mechanics (see \cite{BF}), experienced a substantial breakthrough in 2006 with the introduction of the $S$-spectrum and the associated $S$-resolvent operators. These notions have played a crucial role in the development of the theory, enabling a wide range of applications, including fractional diffusion problems (see \cite{ColomboDenizPinton2020, ColomboDenizPinton2021, FJBOOK, ColomboPelosoPinton2019}), Schur analysis (see \cite{ACS2016}), and the study of characteristic operator functions with applications to linear systems theory (see \cite{AlpayColSab2020}).

 Conversely, the Cauchy formula for axially monogenic functions leads to the monogenic functional calculus, introduced and developed by McIntosh and his collaborators, see \cite{JBOOK, JM}, which relies on a different notion of spectrum: the monogenic spectrum.

An alternative way to define the monogenic functional calculus is through the integral version of the Fueter-Sce mapping theorem, see Theorem \ref{FSinte}. This leads to the so-called $F$-functional calculus, see \cite{FABJONA}. The construction is summarized in the following diagram:

 	\begin{equation}
	\label{starr}
	{\footnotesize
		\begin{CD}
			{\mathcal{SH}(U)} @.  {\mathcal{AM}(U)} \\   @V  VV
			@.
			\\
			{{\rm  Slice\ Cauchy \ Formula}}  @> T_{FS2}=\Delta_{n+1}^{h_n}>> {{\rm Fueter-Sce\ theorem \ in \ integral\ form}}
			\\
			@V VV    @V VV
			\\
			{S-{\rm  functional \ calculus}} @. F-{{\rm functional \ calculus}}
		\end{CD}
	}
\end{equation}

\begin{remark}
In the above diagram, we have omitted the notation $\Omega_\Pi$, since our focus is solely on the mapping $T_{FS2}=\Delta_{n+1}^{h_n}$ from $\mathcal{SH}(U)$ to $\mathcal{AM}(U)$. Hence, in what follows, we will consider axially symmetric open sets, which will be denoted by $U$.
\end{remark}

\section{Factorizations of the Fueter-Sce map}

Following the Riesz--Dunford functional calculus, see \cite{RD}, where the resolvent equation yields a product rule, one is naturally led to search for an analogous resolvent equation in the setting of the $F$-functional calculus. This problem was investigated in \cite{FABJONA} and, more generally, in \cite{CDS1, CDS}. We remark that the situation for the $F$-functional calculus is more delicate, since it is defined by means of an integral transform. The following notion plays a key role in our study, as it is essential in establishing the product rule for the $F$-functional calculus, see \cite{CDP25}.  

\begin{definition}
	We refer to the fine structures on the $S$-spectrum as the collection of functions together with the corresponding functional calculi induced by the factorization of the second Fueter--Sce mapping.
\end{definition}

The study of the factorization of the Fueter--Sce mapping began in the quaternionic setting. In this case, the first factorization considered is 
$
\Delta_4 = \overline{D}D.
$
This leads to the following refinement of the scheme \eqref{SC}:

\begin{equation}
	\label{fine1} \mathcal{SH}_L(U) \overset{D}{\longrightarrow} \mathcal{AH}_L(U) \overset{\overline{D}}{\longrightarrow} \mathcal{AM}_L(U),
\end{equation}  

where  
$$
\mathcal{AH}_L(U):=D \left(\mathcal{SH}_L(U)\right)= \{ Df \, | \, f \in \mathcal{SH}_L(U)\},
$$
is the set of axially harmonic functions. In \cite[Thm.~ 4.1]{CDPS}, the application of the operator $D$ to the second form of the left slice hyperholomorphic Cauchy kernel leads to the following result:  

\begin{equation}
	D(S^{-1}_L(s,q))=-2\mathcal{Q}_{c,s}^{-1}(q), \qquad s \notin [q],
\end{equation}
where $q=x_0+x_1 e_1+x_2 e_2+x_3 e_1 e_2$ is a quaternion.

In \cite{Polyf1, Polyf2}, the factorization $\Delta_4 = D \overline{D}$ was taken into consideration. This gives rise to the following diagram:  

\begin{equation} 
	\mathcal{SH}_L(U) \overset{\overline{D}}{\longrightarrow} \mathcal{AP}_2^L(U) \overset{D}{\longrightarrow} \mathcal{AM}_L(U),
\end{equation}  
where  
$$
\mathcal{AP}_2^L(U) := \overline{D} \bigl(\mathcal{SH}_L(U)\bigr) = \{ \overline{D}f \;|\; f \in \mathcal{SH}_L(U)\},
$$
is the set of axially polyanalytic functions of order $2$. The application of the conjugate Fueter operator to the second form of the left slice hyperholomorphic Cauchy kernel yields  

\begin{equation}
	\overline{D}\!\left(S^{-1}_L(s,q)\right) = -F_L^3(s,q)s + q_0 F_L^3(s,q), \qquad s \notin [q],
\end{equation}  
see \cite[Thm.~4.1]{Polyf1}. In \cite{Fivedim} we start considering the Clifford case specifically the case $n=5$. In this case the following factorizations of the Fueter-Sce maps (i.e. $\Delta^2_6$) were considered
\begin{equation}
\label{oper}
D, \quad \Delta_6, \quad \Delta_6 D, \quad \overline{D}, \quad \overline{D}^2, \quad D^2, \quad \Delta_6 \overline{D}.
\end{equation}
These can be summarized in the following diagram
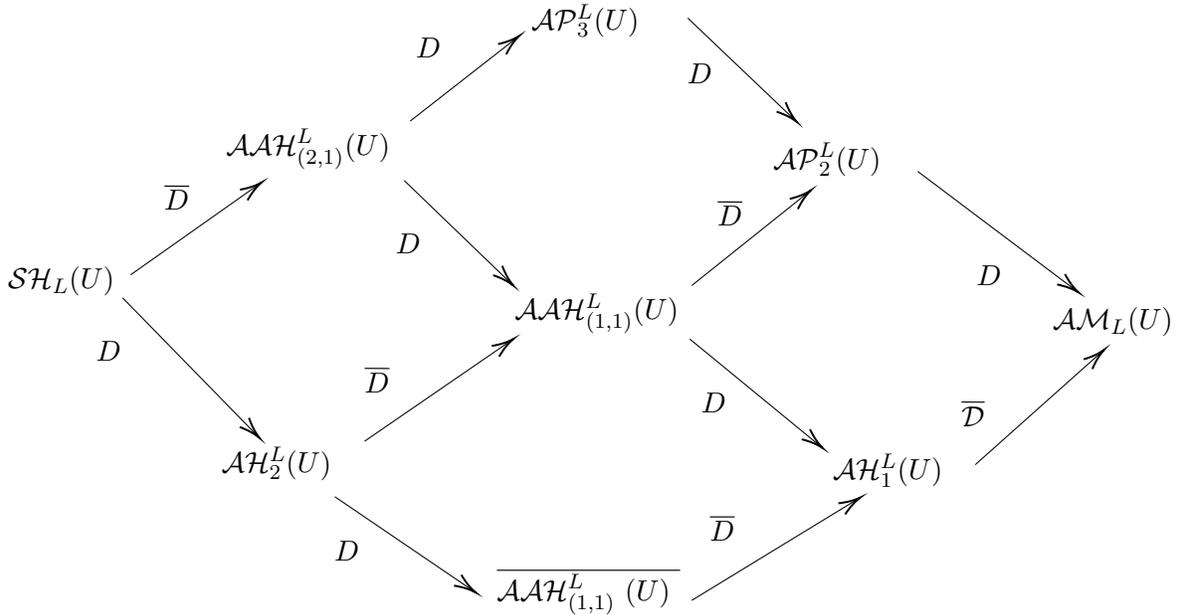
\newline
\begin{center}
\begin{tikzpicture}[x=0.75pt,y=0.75pt,yscale=-1,xscale=1]
	%uncomment if require: \path (0,400); %set diagram left start at 0, and has height of 400
	
	%Straight Lines [id:da6312392585715721] 
	%\draw    (61.22,178.04) -- (109.64,178.35) ;
	%\draw [shift={(111.64,178.37)}, rotate = 180.37] [color={rgb, 255:red, 0; green, 0; blue, 0 }  ][line width=0.75]    (10.93,-3.29) .. controls (6.95,-1.4) and (3.31,-0.3) .. (0,0) .. controls (3.31,0.3) and (6.95,1.4) .. (10.93,3.29)   ;
	%Straight Lines [id:da66684324315636] 
	\draw    (192.08,174.42) -- (256.45,128.99) ;
	\draw [shift={(258.08,127.83)}, rotate = 144.79] [color={rgb, 255:red, 0; green, 0; blue, 0 }  ][line width=0.75]    (10.93,-3.29) .. controls (6.95,-1.4) and (3.31,-0.3) .. (0,0) .. controls (3.31,0.3) and (6.95,1.4) .. (10.93,3.29)   ;
	%Straight Lines [id:da48749620752004796] 
	\draw    (310.08,258.42) -- (383.44,207.97) ;
	\draw [shift={(385.08,206.83)}, rotate = 145.48] [color={rgb, 255:red, 0; green, 0; blue, 0 }  ][line width=0.75]    (10.93,-3.29) .. controls (6.95,-1.4) and (3.31,-0.3) .. (0,0) .. controls (3.31,0.3) and (6.95,1.4) .. (10.93,3.29)   ;
	%Straight Lines [id:da26101923227832136] 
	\draw    (295.08,286.42) -- (364.43,333.62) ;
	\draw [shift={(366.08,334.75)}, rotate = 214.25] [color={rgb, 255:red, 0; green, 0; blue, 0 }  ][line width=0.75]    (10.93,-3.29) .. controls (6.95,-1.4) and (3.31,-0.3) .. (0,0) .. controls (3.31,0.3) and (6.95,1.4) .. (10.93,3.29)   ;
	%Straight Lines [id:da20008418071317702] 
	\draw    (474.08,206.83) -- (539.54,260.56) ;
	\draw [shift={(541.08,261.83)}, rotate = 219.38] [color={rgb, 255:red, 0; green, 0; blue, 0 }  ][line width=0.75]    (10.93,-3.29) .. controls (6.95,-1.4) and (3.31,-0.3) .. (0,0) .. controls (3.31,0.3) and (6.95,1.4) .. (10.93,3.29)   ;
	%Straight Lines [id:da9291501482728126] 
	\draw    (475.08,179.83) -- (533.52,133.08) ;
	\draw [shift={(535.08,131.83)}, rotate = 141.34] [color={rgb, 255:red, 0; green, 0; blue, 0 }  ][line width=0.75]    (10.93,-3.29) .. controls (6.95,-1.4) and (3.31,-0.3) .. (0,0) .. controls (3.31,0.3) and (6.95,1.4) .. (10.93,3.29)   ;
	%Straight Lines [id:da9930084062340787] 
	\draw    (333.08,96.83) -- (386.51,55.07) ;
	\draw [shift={(388.08,53.83)}, rotate = 141.98] [color={rgb, 255:red, 0; green, 0; blue, 0 }  ][line width=0.75]    (10.93,-3.29) .. controls (6.95,-1.4) and (3.31,-0.3) .. (0,0) .. controls (3.31,0.3) and (6.95,1.4) .. (10.93,3.29)   ;
	%Straight Lines [id:da6806772901684502] 
	\draw    (473.08,44.83) -- (526.64,95.87) ;
	\draw [shift={(528.08,97.25)}, rotate = 223.62] [color={rgb, 255:red, 0; green, 0; blue, 0 }  ][line width=0.75]    (10.93,-3.29) .. controls (6.95,-1.4) and (3.31,-0.3) .. (0,0) .. controls (3.31,0.3) and (6.95,1.4) .. (10.93,3.29)   ;
	%Straight Lines [id:da9602771466745729] 
	\draw    (589.08,122.25) -- (668.49,182.05) ;
	\draw [shift={(670.08,183.25)}, rotate = 216.98] [color={rgb, 255:red, 0; green, 0; blue, 0 }  ][line width=0.75]    (10.93,-3.29) .. controls (6.95,-1.4) and (3.31,-0.3) .. (0,0) .. controls (3.31,0.3) and (6.95,1.4) .. (10.93,3.29)   ;
	%Straight Lines [id:da7876537735653709] 
	\draw    (475.08,338.75) -- (558.38,287.88) ;
	\draw [shift={(560.08,286.83)}, rotate = 148.58] [color={rgb, 255:red, 0; green, 0; blue, 0 }  ][line width=0.75]    (10.93,-3.29) .. controls (6.95,-1.4) and (3.31,-0.3) .. (0,0) .. controls (3.31,0.3) and (6.95,1.4) .. (10.93,3.29)   ;
	%Straight Lines [id:da03257555248954214] 
	\draw    (188,186) -- (255.69,255.32) ;
	\draw [shift={(257.08,256.75)}, rotate = 225.68] [color={rgb, 255:red, 0; green, 0; blue, 0 }  ][line width=0.75]    (10.93,-3.29) .. controls (6.95,-1.4) and (3.31,-0.3) .. (0,0) .. controls (3.31,0.3) and (6.95,1.4) .. (10.93,3.29)   ;
	%Straight Lines [id:da43600262909575405] 
	\draw    (330.08,126.83) -- (383.64,178.45) ;
	\draw [shift={(385.08,179.83)}, rotate = 223.94] [color={rgb, 255:red, 0; green, 0; blue, 0 }  ][line width=0.75]    (10.93,-3.29) .. controls (6.95,-1.4) and (3.31,-0.3) .. (0,0) .. controls (3.31,0.3) and (6.95,1.4) .. (10.93,3.29)   ;
	%Straight Lines [id:da85649492230783] 
	\draw    (618.08,270.25) -- (681.59,213.58) ;
	\draw [shift={(683.08,212.25)}, rotate = 138.26] [color={rgb, 255:red, 0; green, 0; blue, 0 }  ][line width=0.75]    (10.93,-3.29) .. controls (6.95,-1.4) and (3.31,-0.3) .. (0,0) .. controls (3.31,0.3) and (6.95,1.4) .. (10.93,3.29)   ;
	
	% Text Node
	%\draw (18.09,165.97) node [anchor=north west][inner sep=0.75pt]  [rotate=-0.38] [align=left] {$\displaystyle \mathcal{O}( D)$};
	% Text Node
	%\draw (69.72,155.35) node [anchor=north west][inner sep=0.75pt]  [rotate=-0.38]  {$T_{FS1}$};
	% Text Node
	\draw (129.07,168.79) node [anchor=north west][inner sep=0.75pt]  [rotate=-0.38] [align=left] {$\mathcal{SH}_L(U)$};
	% Text Node
	\draw (173.64,206.53) node [anchor=north west][inner sep=0.75pt]  [rotate=-0.38]  {$D$};
	% Text Node
	\draw (207.89,127.48) node [anchor=north west][inner sep=0.75pt]  [rotate=-0.38]  {$\overline{D}$};
	% Text Node
	\draw (236.84,259.59) node [anchor=north west][inner sep=0.75pt]  [rotate=-0.38]  {$\mathcal{AH}^L_2( U)$};
	% Text Node
	\draw (239,99.4) node [anchor=north west][inner sep=0.75pt]   {$\mathcal{AAH}_{(2,1)}^L(U) \ $};
	% Text Node
	\draw (308.89,219.48) node [anchor=north west][inner sep=0.75pt]  [rotate=-0.38]  {$\overline{D}$};
	% Text Node
	\draw (293.64,307.53) node [anchor=north west][inner sep=0.75pt]  [rotate=-0.38]  {$D$};
	% Text Node
	\draw (384.58,182.4) node [anchor=north west][inner sep=0.75pt]    {$\mathcal{AAH}_{(1,1)}^L(U) \ $};
	% Text Node
	\draw (375.58,321.4) node [anchor=north west][inner sep=0.75pt]    {$\overline{\mathcal{AAH}_{(1,1)}^L \ (U) \ }$};
	% Text Node
	\draw (478.64,232.53) node [anchor=north west][inner sep=0.75pt]  [rotate=-0.38]  {$D$};
	% Text Node
	\draw (486.89,135.48) node [anchor=north west][inner sep=0.75pt]  [rotate=-0.38]  {$\overline{D}$};
	% Text Node
	\draw (545,263.4) node [anchor=north west][inner sep=0.75pt]    {$\mathcal{AH}^L_1(U)$};
	% Text Node
	\draw (515,106.4) node [anchor=north west][inner sep=0.75pt]    {$\mathcal{AP}_{2}^L(U)$};
	% Text Node
	\draw (393,36.4) node [anchor=north west][inner sep=0.75pt]    {$\mathcal{AP}_{3}^L( U)$};
	% Text Node
	\draw (334.89,53.48) node [anchor=north west][inner sep=0.75pt]  [rotate=-0.38]  {$D$};
	% Text Node
	\draw (656,189.48) node [anchor=north west][inner sep=0.75pt]    {$\mathcal{AM}_L( U)$};
	% Text Node
	\draw (617.64,168.53) node [anchor=north west][inner sep=0.75pt]  [rotate=-0.38]  {$D$};
	% Text Node
	\draw (609.89,236.48) node [anchor=north west][inner sep=0.75pt]  [rotate=-0.38]  {$\mathcal{\overline{D}}$};
	% Text Node
	\draw (483,294.48) node [anchor=north west][inner sep=0.75pt]    {$\overline{D}$};
	% Text Node
	\draw (471.64,66.53) node [anchor=north west][inner sep=0.75pt]  [rotate=-0.38]  {$D$};
	% Text Node
	\draw (324.64,152.53) node [anchor=north west][inner sep=0.75pt]  [rotate=-0.38]  {$D$};

\end{tikzpicture}
\end{center}
where
$$ \mathcal{AP}_{3-\ell}^L(U)= \{ \Delta^\ell_{6} \overline{D}^{2-\ell}f\, | \, f \in \mathcal{SH}_L(U) \}, \qquad \ell=0,1,$$
is the set of axially polyanalytic functions of order 2 and 3,
$$ \mathcal{AH}_{3-\ell}^L(U)= \{D \Delta_6^{\ell-1}f\, | \, f \in \mathcal{SH}_L(U) \}, \qquad \ell=1,2,$$
is the set of axially harmonic and bi-harmonic functions,
$$ \mathcal{AAH}_{(2-\gamma,1)}^L(U)= \{D^\gamma \overline{D}f \, | \, f \in \mathcal{SH}_L(U)\}, \qquad \gamma=0,1, \quad \overline{\mathcal{AAH}^L_{(1,1)}(U)}:= \{D^2f\, | \, f \in \mathcal{SH}_L(U)\},$$
is the set of axially analytic harmonic functions of type $(2-\gamma,1)$ and axially anti-analytic-Harmonic functions of type $(1,1)$. For $s, x \in \mathbb{R}^{n+1}$ with $s \notin [x]$, the action of the operators in \eqref{oper} on the second form of the left slice hyperholomorphic Cauchy kernel yields

\begin{equation}
D(S^{-1}_L(s,x))=-4 \mathcal{Q}_{c,s}^{-1}(x), \qquad \Delta_6 D(S^{-1}_L(s,x))=16 \mathcal{Q}_{c,s}^{-2}(x),
\end{equation}
see \cite[Thm. 6.10 and Thm. 6.12]{Fivedim},
\begin{eqnarray}
	\nonumber
&& \Delta_6(S^{-1}_L(s,x))=8(s-\bar{x}) \mathcal{Q}_{c,s}^{-2}(x), \qquad \overline{D}(S^{-1}_L(s,x))=4(s-\bar{x}) \mathcal{Q}_{c,s}^{-2}(x)(s-x_0)+2 \mathcal{Q}_{c,s}^{-1}(x)\\ && \qquad \qquad \qquad D^2(S^{-1}_L(s,x))=	16 \mathcal{Q}_{c,s}^{-2}(x)(s-x_0)-8(s-\bar{x}) \mathcal{Q}_{c,s}^{-2}(x).
\end{eqnarray}

see \cite[Thm. 6.11, Thm. 6.13 and Thm. 6.15]{Fivedim}

\begin{equation}
\Delta_6 \overline{D}(S^{-1}_L(s,x))=-64 (s-\bar{x}) \mathcal{Q}_{c,s}^{-3}(x) (s-x_0), \qquad \overline{D}^2 (S^{-1}(s,x))=32 (s-\bar{x}) \mathcal{Q}_{c,s}^{-3}(x)(s-x_0)^3,
\end{equation}
see \cite[Thm. 6.13 and Thm. 6.14]{Fivedim}. In \cite{CDP25}, we studied the most general factorizations of the Fueter-Sce map in the Clifford setting for odd $n$. In particular, we proved that suitable factorizations of $\Delta_{n+1}^{\frac{n-1}{2}}$ give rise to the following refinement of the scheme \eqref{SC}:

\begin{equation}
\label{polyh}
	\begin{CD}
		&& \textcolor{black}{\mathcal{SH}_L(U)}  @>\ \    D \Delta^{m-1}_{n+1}>>\textcolor{black}{\mathcal{APH}^L_{h_n-m+1}(U)}@>\ \   \overline{D} \Delta_{n+1}^{h_n-m}>>\textcolor{black}{\mathcal{AM}_L(U)}, \qquad 1 \leq m \leq h_n,
	\end{CD}
\end{equation}
where
$$\mathcal{APH}^L_{h_n-m+1}(U)=\{D\Delta^{m-1}_{n+1}f \ :\ f\in \mathcal{SH}_L(U)\},$$
is the set of axially polyharmonic functions of order $h_n-m+1$. In \cite[Theorem 4.16]{CDP25}, the authors proved that the application of the operator $D\Delta^{m-1}_{n+1}$ leads to the following:
\begin{equation}
\label{new1}
D\Delta^{m-1}_{n+1}(S^{-1}_L(s,x))=\sigma_{n,m} \mathcal{Q}_{c,s}^{-m}(x),           \qquad s \notin [x],
\end{equation}
where $\sigma_{n,m}=2^{2m-1}(m-1)! (-h_n)_m$. Another interesting factorization of the Fueter-Sce map lead to the following scheme 
\begin{equation}
\label{cliff}
	\begin{CD}
		&& \textcolor{black}{\mathcal{SH}_L(U)}  @>\ \   \Delta_{n+1}^m >>\textcolor{black}{\mathcal{AAH}_{1,h_n-m}^L(U)}@>\ \
		\Delta ^{h_n-m}_{n+1}>>\textcolor{black}{\mathcal{AM}_L(U)}, \quad 1 \leq m \leq h_n-1
	\end{CD}
\end{equation}
where 
$$
\mathcal{AAH}^L_{(h_n-m,1)}(U)=\{ \Delta_{n+1}^{m} f(x) \ :\ f\in \mathcal{SH}_L(U)\}.
$$
is the set of axially analytic Harmonic functions of order $(h_n-m,1)$. In this case the application of the operator $\Delta_{n+1}^{m}$ to the second form of the slice hyperholomorphic Cauchy kernel lead to
\begin{equation}
\label{appL}
\Delta_{n+1}^{m}(S^{-1}_L(s,x))=\gamma_m (s-\bar{x}) \mathcal{Q}_{c,s}^{-m-1}(x),\qquad s \notin [x],
\end{equation}
where $\gamma_m=4^m m! (-h_n)_m $. The other possible factorizations of the operator $\Delta_{n+1}^{\frac{n-1}{2}}$ lead to the following factorization of \eqref{SC}:
\begin{equation}
	\begin{CD}
		&& \textcolor{black}{\mathcal{SH}_L(U)}  @>\ \    \overline{D}^{h_n-m} \Delta^{\ell}_{n+1}>>\textcolor{black}{\mathcal{AP}_{h_n-m+1}^L(U)}@>\ \   D^{h_n-m}>>\textcolor{black}{\mathcal{AM}_L(U)}, \quad 0 \leq \ell \leq h_n-1,
	\end{CD}
\end{equation}
where
$$
\mathcal{AP}_{h_n-m+1}^L(U)=\{\Delta_{n+1}^{m} \overline{D}^{h_n- m}f\ :\ f\in \mathcal{SH}_L(U)\}
$$
is the set of axially polyanalytic functions of order $h_n-m+1$. In this case the application of the operator $\Delta_{n+1}^{m} \overline{D}^{h_n- m}$ to the second form of the slice hyperholomorphic Cauchy kernel gives
\begin{equation}
\label{polyapp}
\Delta_{n+1}^{m} \overline{D}^{h_n- m}(S^{-1}_L(s,x))=\frac{(-1)^{h_n-\ell}}{(h_n-\ell)!} F_L^n(s,x)(s-x_0)^{h_n-\ell}, \qquad s \notin [x],
\end{equation} 
see \cite[Thm. 8.10]{CDP25}.

\begin{remark}
The factorizations \eqref{polyh} and \eqref{cliff} has been used in \cite{CDP25} to deduce a product for the $F$-functional calculus in the Clifford setting. 
\end{remark}

\begin{remark}
The sets of polyharmonic, polyanalytic, and Cliffordian functions have been extensively studied in the literature. More precisely, polyharmonic functions were considered in \cite{Aro}. Polyanalytic functions were first studied in relation to elasticity theory (see \cite{Russ}), and later a rigorous study of this class of functions was carried out in \cite{Balk}. Moreover, in \cite{FISH} the theory of polyanalytic functions found applications in signal analysis. The class of analytic harmonic functions, introduced and studied by G.~Laville and I.~Ramadanoff (see \cite{LR, LL}), was given the name of Cliffordian functions.

\end{remark} 

In this paper, our goal is to determine the most general possible factorizations of the Fueter--Sce map and their applications to the second form of the left slice hyperholomorphic Cauchy kernel. Generally speaking, the Fueter--Sce map can be expressed in terms of the Dirac operator and its conjugate as follows:
\begin{equation}
	\label{fact}
	\Delta_{n+1}^{\frac{n-1}{2}}
	= \underbrace{D \cdots D}_{n-1 \, \text{times}} \;
	\underbrace{\overline{D} \cdots \overline{D}}_{n-1 \, \text{times}}.
\end{equation}

Thus, we have two possible factorizations:
\begin{itemize}
	\item If we consider more conjugate Dirac operators than Dirac operators in \eqref{fact}, we obtain the factorization
\begin{equation}
\label{FF0}
	\overline{D}^{\beta} \, \Delta_{n+1}^{m}, 
\qquad m \in \mathbb{N}, \quad \beta + m \leq h_n.
\end{equation}
	\item On the other hand, if we consider more Dirac operators than conjugate Dirac operators in \eqref{fact}, we obtain the factorization
\begin{equation}
\label{FF1}
	D^{\beta} \, \Delta_{n+1}^{m}, 
\qquad m \in \mathbb{N}, \quad \beta + m \leq h_n.
\end{equation}
\end{itemize}

Now, we want to apply the operators \eqref{FF0} and \eqref{FF1} to the second form of the  slice hyperholomorphic Cauchy kernel. Before doing so, we need the following technical result.

\begin{remark}
We have presented the results on the factorization of the Fueter-Sce map only in the setting of left slice hyperholomorphic functions. However, analogous results also hold for right slice hyperholomorphic functions. Throughout this paper, we restrict ourselves to the left slice hyperholomorphic case, since the corresponding statements for the right case follow from similar computations.
\end{remark}

	\begin{lemma}
		\label{l_one}
Let $n$ be an odd integer and set $h_n=\frac{n-1}{2}$. Denote by $D$ the Dirac operator with respect to the variable $x$. For $s, x \in \mathbb{R}^{n+1}$ with $s \notin [x]$, we have
	\begin{align}
	D\left( (s-\bar x) \mathcal{Q}_{c,s}^{-m}(x) \right) & =-2(h_n-m+1) \mathcal{Q}_{c,s}^{-m}(x) \label{d_cauchy_ker}\\
	D\left( \mathcal{Q}_{c,s}^{-m}(x) \right) & =4m (s-x_0)\mathcal{Q}_{c,s}^{-m-1}(x)-2m(s-\bar x) \mathcal{Q}^{-m-1}_{c,s}(x)  \label{d2_cauchy_ker}\\
	D\left( (s-x_0)^k \mathcal{Q}_{c,s}^{-m}(x) \right) &=4m (s-x_0)^{k+1}\mathcal{Q}_{c,s}^{-m-1}(x)-2m(s-\bar x) \mathcal{Q}^{-m-1}_{c,s}(x) (s-x_0)^k \nonumber\\
	& \, \, \, \,\, \,-k(s-x_0)^{k-1} \mathcal{Q}_{c,s}^{-m}(x)  \label{d3_cauchy_ker}\\
	D\left( (s-\bar x) \mathcal{Q}_{c,s}^{-m}(x) (s-x_0)^k \right) & =2(m-h_n-1) \mathcal{Q}_{c,s}^{-m} (x)(s-x_0)^k \nonumber\\
	&\, \, \, \,\, \,-k(s-\bar x)\mathcal{Q}_{c,s}^{-m}(x)(s-x_0)^{k-1} \label{d4_cauchy_ker}
\end{align}
	\end{lemma}
	\begin{proof}
		First we prove \eqref{d_cauchy_ker}. We observe that 
			\begingroup\allowdisplaybreaks
		\begin{align*}
			D\left( (s-\bar x) \mathcal{Q}_{c,s}^{-m}(x) \right) & = \partial_{x_0}(\left( (s-\bar x) \mathcal{Q}_{c,s}^{-m}(x) \right))+\sum_{i=1}^n e_i \partial_{x_i} \left( (s-\bar x) \mathcal{Q}_{c,s}^{-m}(x) \right) \\
			&= - \mathcal{Q}_{c,s}^{-m}(x)+2m(s-\bar x) (s-x_0) \mathcal{Q}_{c,s}^{-m-1}(x)\\
			& +  \sum_{i=1}^n e_i \left( e_i \mathcal{Q}_{c,s}^{-m}(x) -2m x_i (s-\bar x) \mathcal{Q}_{c,s}^{-m-1}(x) \right)\\
			&= (-1-n) \mathcal{Q}_{c,s}^{-m}(x) +2m \left( (s-\bar x)s -  x(s-\bar x)  \right) \mathcal{Q}_{c,s}^{-m-1}(x)\\
			&= 2\left (  m-\frac {n+1}{2} \right )\mathcal{Q}_{c,s}^{-m}(x)= 2\left (  m-h_n -1 \right )\mathcal{Q}_{c,s}^{-m}(x).
		\end{align*}
		\endgroup
		Now we prove \eqref{d2_cauchy_ker}. We have  
		\begin{align*}
			D\left( \mathcal{Q}_{c,s}^{-m}(x) \right) & = \partial_{x_0}\left( \mathcal{Q}_{c,s}^{-m}(x) \right)+\sum_{i=1}^n e_i \partial_{x_i} \left( \mathcal{Q}_{c,s}^{-m}(x) \right)\\
			 &= 2m (s-x_0) \mathcal{Q}_{c,s}^{-m-1}(x)-2\sum_{i=1}^n e_i m x_i \mathcal{Q}_{c,s}^{-m-1}(x)\\
			& = 2m(s-x)\mathcal{Q}_{c,s}^{-m-1}(x)\\
			&= 2m(s-\bar x) \mathcal{Q}_{c,s}^{-m-1}(x)-4m \unx \mathcal{Q}_{c,s}^{-m-1}(x) \\
			& = 2m(s-\bar x) \mathcal{Q}_{c,s}^{-m-1}(x)+4m (s-x_0+x_0 - \unx -s) \mathcal{Q}_{c,s}^{-m-1}(x)\\
			&=  2m(s-\bar x) \mathcal{Q}_{c,s}^{-m-1}(x) -4m (s-\bar x) \mathcal{Q}_{c,s}^{-m-1}(x) +4m \mathcal{Q}_{c,s}^{-m-1}(x) (s-x_0)\\
			&= 4m \mathcal{Q}_{c,s}^{-m-1}(x) (s-x_0) -2m(s-\bar x) \mathcal{Q}_{c,s}^{-m-1}(x) .
		\end{align*}
Formulas \eqref{d3_cauchy_ker} and \eqref{d4_cauchy_ker} can be derived from \eqref{d_cauchy_ker} and \eqref{d2_cauchy_ker}, respectively, using the Leibniz rule. In fact, we have 
\begin{align*} D\left ( \mathcal{Q}_{c,s}^{-m}(x) (s-x_0)^k \right) & =  \partial_{x_0}( \mathcal{Q}_{c,s}^{-m}(x)) (s-x_0)^k + \mathcal{Q}_{c,s}^{-m}(x) \partial_{x_0} (s-x_0)^k\\
& +\sum_{i=1}^n e_i\partial_{x_i}(\mathcal Q^{-m}_{c,s}(x)) (s-x_0)^k + e_i \mathcal Q_{c,s}(x) \partial_{x_i}((s-x_0)^k)\\
&=D\left( \mathcal{Q}^{-m}_{c,s}(x) \right) (s-x_0)^k + \mathcal{Q}^{-m}_{c,s}(x) D\left( (s-x_0)^k\right)
\end{align*}
and, thus, we obtain

		\begin{align*}
			& D\left( \mathcal{Q}_{c,s}^{-m}(x) (s-x_0)^k \right)  =D\left( \mathcal{Q}^{-m}_{c,s}(x) \right) (s-x_0)^k + \mathcal{Q}^{-m}_{c,s}(x) D\left( (s-x_0)^k\right)\\
			& = [4m \mathcal{Q}^{-m-1}_{c,s}(x)(s-x_0) - 2m (s- \bar x)\mathcal{Q}_{c,s}^{-m-1}(x) ] (s-x_0)^k -k \mathcal{Q}_{c,s}^{-m}(x) (s-x_0)^{k-1}\\
			&= 4m \mathcal{Q}^{-m-1}_{c,s}(x)(s-x_0) (s-x_0)^{k+1} - 2m (s- \bar x) \mathcal{Q}^{-m-1}_{c,s}(x) (s-x_0)^k -k \mathcal{Q}_{c,s}^{-m}(x) (s-x_0)^{k-1}.
		\end{align*}  	 
		Moreover
		\begin{align*}
			D\left( (s-\bar x) \mathcal{Q}_{c,s}^{-m}(x) (s-x_0)^k \right) & =D\left( (s-\bar x) \mathcal{Q}_{c,s}^{-m}(x) \right) (s-x_0)^k + (s-\bar x) \mathcal{Q}_{c,s}^{-m}(x) D\left( (s-x_0)^k\right)\\
			& = 2(m-h_n-1) \mathcal{Q}_{c,s}^{-m}(x) (s-x_0)^k -k(s-\bar x) \mathcal{Q}_{c,s}^{-m}(x) (s-x_0)^{k-1}.
		\end{align*}  	 
	\end{proof}
By combining the above result with the identity $D + \overline{D} = 2 \,\partial_{x_0}$, one obtains the following result:

	\begin{lemma}
		\label{l_one1}
Let $n$ be an odd integer and set $h_n=\frac{n-1}{2}$. Denote by $D$ the Dirac operator with respect to the variable $x$. For $s, x \in \mathbb{R}^{n+1}$ with $s \notin [x]$, we have
		\begin{align*}
			\overline D\left( (s-\bar x) \mathcal{Q}_{c,s}^{-\ell}(x) \right) & =2(h_n-\ell) \mathcal{Q}_{c,s}^{-\ell}(x) +4\ell(s-\bar x )(s-x_0) \mathcal{Q}_{c,s}^{-\ell-1}(x) \\
			\overline D\left( \mathcal{Q}_{c,s}^{-\ell}(x) \right) & =2\ell (s-\bar x)\mathcal{Q}_{c,s}^{-\ell-1}(x) \\
			\overline D\left( (s-x_0)^k \mathcal{Q}_{c,s}^{-\ell}(x) \right) &=-k(s-x_0)^{k-1} \mathcal{Q}^{-\ell}_{c,s}(x)+2\ell(s-\bar{x})(s-x_0)^k \mathcal{Q}^{-\ell-1}_{c,s}(x)\\
			\overline D\left( (s-\bar x) \mathcal{Q}_{c,s}^{-\ell}(x) (s-x_0)^k \right) & =2(h_n-\ell) \mathcal{Q}_{c,s}^{-\ell}(x) (s-x_0)^k+4\ell(s-\bar x)\mathcal{Q}_{c,s}^{-\ell-1}(x)(s-x_0)^{k+1} \\
			& -k(s-\bar x)(s-x_0)^{k-1} \mathcal{Q}_{c,s}^{-\ell}(x) \nonumber
		\end{align*}
	\end{lemma}

We now have all the necessary tools to demonstrate the application of the operator \eqref{FF0} to the second form of the left slice hyperholomorphic Cauchy kernel.

	\begin{theorem}\label{p11}
		Let $n$ be a fixed odd number and $h_n:=(n-1)/2$. Let $m$, $\beta \in \mathbb{N}$ be such that $m+\beta\leq h_n$. We assume that $s, x \in \mathbb{R}^{n+1}$ with $s \notin [x]$, and we denote by $D$ and $\Delta_{n+1}$ the Dirac and Laplace operators, respectively, applied with respect to the variable $x$. We set $\gamma_m=4^m m! (-h_n)_m$. If $\beta=2k_1+1$, where $k_1\in\mathbb N$  we have
		\begin{eqnarray}\label{do}
			&& D^{\beta}\left( \Delta^m_{n+1} S^{-1}_L(s,x) \right)= \frac{2^{\beta} (h_n-m) \gamma_m}{m!}  \left( (s-\bar x) \sum_{j=0}^{k_1-1} a^1_{j,k_1,m} \mathcal{Q}_{c,s}^{-m-j-2-k_1}(x) (s-x_0)^{2j+1} \right. \nonumber\\
			&&\left. -\sum_{j=0}^{k_1} b^1_{j,k_1,m}\mathcal{Q}_{c,s}^{-m-1-k_1-j}(x) (s-x_0)^{2j}  \right),
		\end{eqnarray}
		where  
		$$
		a^1_{j,k_1,m}:=2^{2j+1} (m+k_1+1+j)! (k_1-j-1)! \binom{k_1+j}{2j+1} \binom{h_n-m-k_1-j-2}{h_n-m-2k_1-1}
		$$
		and
		$$
		b^1_{j,k_1,m}:= 2^{2j}  (k_1-j)! (m+k_1+j)! \binom{h_n-m-k_1-j-1}{h_n-m-2k_1-1}  \binom{k_1+j}{2j} 
		$$
		If $\beta=2k_2$, where $k_2\in\mathbb N$  we have
		\begin{eqnarray}\label{de}
			&& D^{\beta} \left( \Delta^m_{n+1} S^{-1}_L(s,x)\right )= \frac{2^{\beta} (h_n-m) \gamma_m}{m!} \left( (s-\bar x) \sum_{j=0}^{k_2-1} a^2_{j,k_2,m} \mathcal{Q}_{c,s}^{-m-j-1-k_2}(x) (s-x_0)^{2j}\right.\nonumber \\
			&&\left. -\sum_{j=0}^{k_2-1} b^2_{j,k_2,m} \mathcal{Q}_{c,s}^{-m-1-k_2-j}(x) (s-x_0)^{2j+1}  \right),
		\end{eqnarray}
		where 
		$$
		a^2_{j,k_2,m}:=2^{2j} (m+k_2+j)! (k_2-j-1)! \binom{k_2+j-1}{2j} \binom{h_n-m-k_2-1-j}{h_n-m-2k_2} 
		$$
		and 
		$$
		b^2_{j,k_2,m}:=2^{2j+1} (k_2-j-1)! (m+k_2+j)! \binom{h_n-m-k_2-1-j}{h_n-m-2k_2}\binom{k_2+j}{2j+1}
		$$
	\end{theorem}
	
	\begin{remark}
We have summarized the properties of the coefficients $a^1_{j,k_1,m}$, $b^1_{j,k_1,m}$, $a^2_{j,k_1,m}$, and $b^2_{j,k_1,m}$ that are needed for the proof of Theorem \ref{p11} in the Appendix.
	\end{remark}

	\begin{proof}
		We prove the theorem by a double induction over $k_1 \in \mathbb{N}_0$, $k_2\in\mathbb N$. First, we prove the formulas \eqref{do} and \eqref{de} for $\beta=1$ ($k_1=0$) and $\beta =2$ ($k_2=1$), respectively. By \eqref{appL} and \eqref{d_cauchy_ker} we have
		\[
		D\left( \Delta^m_{n+1} S^{-1}_L(s,x) \right)=\gamma_m D \left( (s-\bar x)\qcsa{-m-1}\right)=-2\gamma_m (h_n-m) \qcsa{-m-1} ,
		\]
		which is formula \eqref{do} when $k_1=0$ if we understood the sum $\sum_{k=0}^{-1}$ to be zero. Moreover, by \eqref{d2_cauchy_ker} we have  
			\begingroup\allowdisplaybreaks
		\begin{align*}
			& D^2\left( \Delta^m_{n+1}S^{-1}_L(s,x) \right)=-2 \gamma_m (h_n-m) D \left( \qcsa{-m-1}\right) \\
			& =-2 \gamma_m (h_n-m) \left( 4(m+1) \qcsa{-m-2}(s-x_0) -2(m+1) (s-\bar x) \qcsa{-m-2} \right)\\
			&= 2^2 \gamma_m (h_n-m) \left( (s-\bar x) (m+1) \qcsa{-m-2} -2(m+1) \qcsa{-m-2}(s-x_0)  \right),
		\end{align*}
		\endgroup
		which is formula \eqref{de} when $k_2=1$. Now we prove that if the formulas \eqref{do} and \eqref{de} hold to be true respectively for $\ku$ and $\kd=\ku+1$ then the formula \eqref{do} holds to be true for $\ku+1$. By \eqref{appL} and the inductive hypothesis we have
			\begingroup\allowdisplaybreaks
		\begin{align*}
			& D^{2 \left (\ku+1 \right) +1} \left (\Delta^m_{n+1} S^{-1}_L(s,x) \right) =\gamma_m D\left( D^{2\kd} \left(\Delta^m_{n+1} S^{-1}_L(s,x)\right) \right)\\
			&=\frac{2^{2k_2} (h_n-m) \gamma_m}{m!} D \left( (s-\bar x) \sum_{j=0}^{\kd-1} a^2_{j, k_2, m} \mathcal{Q}_{c,s}^{-m-j-1-\kd}(x) (s-x_0)^{2j} \right. \\
			& \left. -\sum_{j=0}^{k_2-1} b^2_{j,k_2,m} \mathcal{Q}_{c,s}^{-m-1-\kd-j}(x) (s-x_0)^{2j+1}  \right).
		\end{align*} 
		\endgroup
		We compute separately the Dirac operator over the two summations. By \eqref{d4_cauchy_ker} we have
			\begingroup\allowdisplaybreaks
		\begin{align}\label{f1}
			D & \left( (s-\bar x) \sum_{j=0}^{\kd-1} a^2_{j,k_2,m} \mathcal{Q}_{c,s}^{-m-j-1-\kd}(x) (s-x_0)^{2j} \right)  \\
			&=  \sum_{j=0}^{\kd-1}  -2 a^2_{j,k_2,m} (h_n-m-j-k_2) \mathcal{Q}_{c,s}^{-m-j-1-\kd}(x) (s-x_0)^{2j} \nonumber\\
			& + (s-\bar x) \sum_{j=0}^{\kd-2} (-2j-2) a^2_{j+1,k_2,m}  \mathcal{Q}_{c,s}^{-m-j-2-\kd}(x) (s-x_0)^{2j+1}\nonumber
		\end{align}	
		\endgroup
		where in the last summation we replaced $j-1$ with $j$, and, by \eqref{d3_cauchy_ker}, we also have
			\begingroup\allowdisplaybreaks
		\begin{align}\label{f2}
			& D  \left(  \sum_{j=0}^{\kd-1} b^2_{j,k_2,m} \mathcal{Q}_{c,s}^{-m-1-\kd-j}(x) (s-x_0)^{2j+1}  \right) \nonumber\\
			&  = \sum_{j=1}^{\kd} 4 (m+k_2+j) b^2_{j-1,k_2 ,m} \mathcal{Q}_{c,s}^{-m-1-\kd-j}(x)  (s-x_0)^{2j} \\
			& - (s-\bar x) \sum_{j=0}^{\kd-1} 2(m+k_2+j+1) b^2_{j,k_2,m}  \mathcal{Q}_{c,s}^{-m-2-\kd-j}(x)  (s-x_0)^{2j+1} \nonumber\\
			&- \sum_{j=0}^{\kd-1} (2j+1) b^2_{j,k_2,m}  \mathcal{Q}_{c,s}^{-m-1-\kd-j}(x)  (s-x_0)^{2j}. \nonumber
		\end{align}
		\endgroup
		where in the first summation of \eqref{f2}, we replaced the index $j+1$ with $j$. Now we subtract the terms in \eqref{f2} to \eqref{f1}, and we gather them into two families: those multiplied by $(s-\bar x)$ and those that are not multiplied by $(s-\bar x)$. Thus, we obtain
			\begingroup\allowdisplaybreaks
		\begin{align}
			& \gamma_m D\left( D^{2\kd}  \left(\Delta^m_{n+1} S^{-1}_L(s,x)\right) \right) = \nonumber\\
			& = \frac {2^{2k_2  }  (h_n-m)\gamma_m}{m!}\left( (s-\bar x)\sum_{j=0}^{\kd-2} \left((-2j-2) a^2_{j+1,k_2,m} +2(m+k_2+j+1) b^2_{j,k_2,m} \right) \nonumber \right. \\ 
			&\times \mathcal{Q}_{c,s}^{-m-2-\kd-j}(x)  (s-x_0)^{2j+1} + (s-\bar x)2(m+2k_2)b^2_{k_2-1,k_2,m}  \mathcal{Q}_{c,s}^{-m-1-2\kd}(x)  (s-x_0)^{2k_2-1}  \nonumber\\
			& -( 2a^2_{0,k_2,m}(h_n-m-k_2) - b^2_{0,k_2,m}) \mathcal{Q}_{c,s}^{-m-1-k_2}(x)\nonumber \\
			& -\sum_{j=1}^{k_2-1} \left[ 2 a^2_{j,k_2,m} (h_n-m-j-k_2) + 4(m+k_2+j) b^2_{j-1,k_2,m}-(2j+1) b^2_{j,k_2,m} \right]  \nonumber \\
			& \left.\times \mathcal{Q}_{c,s}^{-m-1-\kd-j}(x)  (s-x_0)^{2j} -4(m+2k_2) b^2_{k_2-1,k_2,m}  \mathcal{Q}_{c,s}^{-m-1-2\kd}(x)  (s-x_0)^{2k_2}  \right).\nonumber
		\end{align}
		\endgroup
		By the properties of the coeffcients given in \eqref{c1}, \eqref{c2}, \eqref{c3}, \eqref{c4} and \eqref{c5} we have
			\begingroup\allowdisplaybreaks
		\begin{align}\label{f3}
			& \gamma_m D\left( D^{2\kd} \left(\Delta_{n+1}^m S^{-1}_L(s,x)\right) \right) = \nonumber\\
			& = \frac{\gamma_m (h_n-m) 2^{2k_{2} +1 }}{m!}\left( (s-\bar x)\sum_{j=0}^{\kd-1} a^1_{j,k_2,m}  \mathcal{Q}_{c,s}^{-m-2-\kd-j}(x)  (s-x_0)^{2j+1} \right. \\
			& \left. -\sum_{j=0}^{k_2} b^1_{j,k_2,m} \qcsa{-m-1-j-k_2} (s-x_0)^{2j}\right)\nonumber.
		\end{align}
		\endgroup
		Now we observe that \eqref{f3} coincide with $\eqref{do}$ with $k_1$ replaced by $k_2$. Since $k_2=k_1+1$ we have proved \eqref{do} with $k_1$ replaced by $k_1+1$. Thus the inductive step for \eqref{do} is proved. 
		\newline
		Now we perform a second step where we prove that if the formulas \eqref{do} hold to be true for $k_1+1$, then the formula \eqref{de} holds to be true for $k_2=k_1+2$.  By the inductive hypothesis, together with formula \eqref{appL} and the fact that $k_1 = k_2 - 1$, we obtain
	\begingroup\allowdisplaybreaks
		\begin{align*}
			&D^{\beta} \left( \Delta^m_{n+1} S^{-1}_L(s,x) \right)=\gamma_m D^{2(k_2+1)}((s-\bar x) \qcsa{-m-1})=\gamma_m D\left( D^{2k_1+3} ((s-\bar x) \qcsa{-m-1}) \right)\nonumber\\
			&= \frac{2^{2k_1+3} (h_n-m) \gamma_m}{m!}  D    \left( (s-\bar x) \sum_{j=0}^{k_1} a^1_{j,k_1+1,m} \mathcal{Q}^{-m-j-3-k_1}_{c,s}(x) (s-x_0)^{2j+1} \right. \\
			& \left. -\sum_{j=0}^{k_1+1} b^1_{j,k_1+1,m} \mathcal{Q}_{c,s}^{-m-2-k_1-j}(x) (s-x_0)^{2j}  \right).\nonumber  
		\end{align*}
		\endgroup
		We compute separately the Dirac operator applied to the two summations. By \eqref{d4_cauchy_ker} we have that
			\begingroup\allowdisplaybreaks
		\begin{align}\label{f4}
			& D \left( (s-\bar x) \sum_{j=0}^{k_1} a^1_{j,k_1+1,m} \mathcal{Q}_{c,s}^{-m-j-3-k_1}(x) (s-x_0)^{2j+1}\right)\nonumber\\
			&  = -\sum_{j=0}^{k_1} 2 (h_n-m-k_1-2-j) a^1_{j,k_1+1,m} \mathcal{Q}_{c,s}^{-m-j-3-k_1}(x) (s-x_0)^{2j+1} \\
			& - (s-\bar x)\sum_{j=0}^{k_1}  (2j+1) a^1_{j,k_1+1,m} \mathcal{Q}_{c,s}^{-m-j-3-k_1}(x) (s-x_0)^{2j} \nonumber
		\end{align} 
		\endgroup
		and, again by \eqref{d3_cauchy_ker}, we obtain
			\begingroup\allowdisplaybreaks
		\begin{align}\label{f5}
			& D \left( \sum_{j=0}^{k_1+1} b^1_{j,k_1+1,m} \mathcal{Q}_{c,s}^{-m-2-k_1-j}(x) (s-x_0)^{2j} \right)\nonumber\\
			&  =\sum_{j=0}^{k_1+1} 4 (m+k_1+j+2) b_{j,k_1+1,m}^1 \qcsa{-m-3-k_1-j} (s-x_0)^{2j+1} \\
			&  - (s-\bar x) \sum_{j=0}^{k_1+1} 2(m+k_1+j+2) b^1_{j,k_1+1,m} \qcsa{-m-3-k_1-j} (s-x_0)^{2j}\nonumber\\
			&\left. - \sum_{j=0}^{k_1} 2(j+1) b^1_{j+1,k_1+1,m} \mathcal{Q}_{c,s}^{-m-3-k_1-j}(x) (s-x_0)^{2j+1} \right). \nonumber
		\end{align} 
		\endgroup
		where in the last summation we have replaced $j-1$ with $j$. Now we add the terms in \eqref{f4} and \eqref{f5}, and we gather them into two families: those multiplied by $(s-\bar x)$ and those that are not multiplied by $(s-\bar x)$. Thus, we obtain
					\begingroup\allowdisplaybreaks
		\begin{align}
			& \gamma_m D\left( D^{2k_1+3}  \left(\Delta^m_{n+1} S^{-1}_L(s,x)\right) \right)= \frac{2^{2k_1+3} (h_n-m)}{m!} \left( (s-\bar x)\right.\nonumber\\
			& \left. \times \sum_{j=0}^{k_1} \left(-(2j+1) a^1_{j,k_1+1,m} + 2(m+k_1+j+2) b^1_{j,k_1+1,m} \right)  \qcsa{-m-3-k_1-j} (s-x_0)^{2j}\right. \\
			& +(s-\bar x) 2 (m+2k_1+3) b^1_{k_1+1,k_1+1,m} \qcsa{-m-2k_1-4} (s-x_0)^{2k_1+2}\nonumber\\
			& - \sum_{j=0}^{k_1} \left( 2(h_n-m-k_1-2-j)a^1_{j,k_1+1,m}+4(m+k_1+j+2)b^1_{j,k_1+1,m}-2(j+1) b^1_{j+1,k_1+1,m}\right) \nonumber \\
			& \times \mathcal{Q}_{c,s}^{-m-3-k_1-j}(x) (s-x_0)^{2j+1}  \nonumber\\
			&  - 4(m+2k_1 + 3) b^1_{k_1+1,k_1+1,m} \mathcal{Q}_{c,s} ^{-m-4-2k_1} (s-x_0)^{2k_1+3}.\nonumber
		\end{align}
	\endgroup
	
		By the properties of the coefficients given in \eqref{C1}, \eqref{C2}, \eqref{C3} and \eqref{C4} we obtain 
		\begin{align}\label{f6}
			& \gamma_m D\left( D^{2k_1+3} \left(\Delta^m_{n+1} S^{-1}_L(s,x)\right) \right)=\frac{2^{2k_1+4} \gamma_m (h_n-m)}{m!} \left( (s-\bar x) \sum_{j=0}^{k_1+1} a^2_{j,k_1+2,m} \qcsa{-m-3-k_1-j}  \right. \nonumber \\ 
			&\left. \times(s-x_0)^{2j} -\sum_{j=0}^{k_1+1} b^2_{j,k_1+2,m} \mathcal{Q}_{c,s}^{-m-3-k_1-j}(x) (s-x_0)^{2j+1} \right).
		\end{align}
		Since $k_1=k_2-1$, we observe that \eqref{f6} coincide with $\eqref{de}$ when $k_2$ is replaced by $k_2+1$. Thus the inductive step for \eqref{de} is proved.
	\end{proof}

		\begin{remark}
		If we set $\beta = 1$ in \eqref{do}, we obtain
		$$
		D (\Delta_{n+1}^m S^{-1}_L(s,x))
		= 2^{2m+1} m! (-h_n)_m \, \mathcal{Q}_{c,s}^{-m-1}(x),
		$$
		which is exactly the same expression as in \eqref{new1}.
		
	\end{remark}

	\begin{remark}
		If we consider the case $ m = h_m $ in Theorem \ref{p11}, we obtain that the application of the operator $D^\beta \Delta^{h_n}_{n+1} $ leads to zero. This is consistent with the Fueter-Sce theorem, since the function $ S^{-1}_L(s,x) $ is slice-hyperholomorphic in $x$.
		
	\end{remark}

By using arguments similar to those in the proof of Theorem \ref{p11}, one can show the application of the operator \eqref{FF1} to the second form of the slice hyperholomorphic Cauchy kernel.

	\begin{theorem}
	\label{second}
		Let $n$ be a fixed odd number and $h_n:=(n-1)/2$. Let $m$, $\beta \in \mathbb{N}$ be such that $m+\beta\leq h_n$. We assume that $s, x \in \mathbb{R}^{n+1}$ with $s \notin [x]$, and we denote by $\overline{D}$ and $\Delta_{n+1}$ the conjugate Dirac and Laplace operators, respectively, applied with respect to the variable $x$. If $\beta=2k_1+1$, where $k_1\in\mathbb N$  we have
		\begin{eqnarray}
			&& \overline D^{\beta}\left( \Delta^m_{n+1} S^{-1}_L(s,x) \right)=  2^{\beta} 4^m  (-h_n)_{m}   \left( (s-\bar x) \sum_{j=0}^{k_1} \mathbf{a}^1_{j,k_1,m} \mathcal{Q}_{c,s}^{-m-j-2-k_1}(x) (s-x_0)^{2j+1} \right. \nonumber\\
						\label{k11}
			&&\left. +\sum_{j=0}^{k_1} \mathbf{b}^1_{j,k_1,m}\mathcal{Q}_{c,s}^{-m-1-k_1-j}(x) (s-x_0)^{2j}  \right),
		\end{eqnarray}
		where
		$$
		\mathbf{a}^1_{j,k_1,m}:=2^{2j+1}  (m +k_1+1+j)!  (k_1-j)! \binom{k_1+j+1}{2j+1} \binom{h_n-m-k_1-j-2}{h_n-m-2k_1-2}
		$$
		and
		$$
		\mathbf{b}^1_{j,k_1,m}:= 2^{2j}  (k_1-j+1)! (m+k_1+j)! \binom{h_n-m-k_1-j-1}{h_n-m-2k_1-2}  \binom{k_1+j}{2j}. 
		$$
		If $\beta=2k_2$, where $k_2\in\mathbb N$ we have
		\begin{eqnarray*}
			&& \overline D^{\beta} \left( \Delta^m_{n+1} S^{-1}_L(s,x)\right )= 2^{\beta} 4^m  (-h_n)_{m} \left( (s-\bar x) \sum_{j=0}^{k_2} \mathbf{a}^2_{j,k_2,m} \mathcal{Q}_{c,s}^{-m-j-1-k_2}(x) (s-x_0)^{2j}\right.\nonumber \\
			&&\left. +\sum_{j=0}^{k_2-1} \mathbf{b}^2_{j,k_2,m} \mathcal{Q}_{c,s}^{-m-1-k_2-j}(x) (s-x_0)^{2j+1}  \right),
		\end{eqnarray*}
		where 
		$$
		\mathbf{a}^2_{j,k_2,m}:=2^{2j} (m+k_2+j) ! (k_2-j)! \binom{k_2+j}{2j} \binom{h_n-m-k_2-1-j}{h_n-m-2k_2-1} 
		$$
		and 
		$$
		\mathbf{b}^2_{j,k_2,m}:=2^{2j+1} (k_2-j)! (m+k_2+j)! \binom{h_n-m-k_2-1-j}{h_n-m-2k_2-1}\binom{k_2+j}{2j+1}
		$$
	\end{theorem}

\begin{remark}
If $\beta = h_n - m$, we recover the kernel in \eqref{polyapp}. We begin by examining the case where $\beta=h_n-m=2k_1+1$. In this case, when $0\leq j<k_1$, the coefficients satisfy the following equations

		$$
	\mathbf{a}^1_{j,k_1,m}:=2^{2j+1}  (m +k_1+1+j)!  \binom{k_1+j+1}{2j+1} \prod_{\ell=0}^{k_1 - j - 1} (h_n - m - 2k_1 - 1 + \ell)=0
	$$
	and, for $0\leq j\leq k_1$, the coefficients satisfy the following equations 
	$$
	\mathbf{b}^1_{j,k_1,m}:= 2^{2j}   (m+k_1+j)!  \binom{k_1+j}{2j} \prod_{\ell=0}^{k_1-j} (h_n -m-2k_1-1+\ell)=0
	$$
	due to the facts
	$
	\frac{(h_n - m - k_1 - j - 2)!}{(h_n - m - 2k_1 - 2)!} = \prod_{\ell=0}^{k_1 - j - 1} (h_n - m - 2k_1 - 1 + \ell),
	$
	and
	$ \frac{(h_n-m-k_1-j-1)!}{(h_n-m-2k_1-2)!}=\prod_{\ell=0}^{k_1-j} (h_n -m-2k_1-1+\ell)$. Since $a^1_{k_1, k_1, m} = 2^{h_n - m} h_n!$, it follows from \eqref{k11} that
	
	\begin{eqnarray*}
&& \overline D^{\beta}\left( \Delta^m_{n+1} S^{-1}_L(s,x) \right)=  2^{h_n-m} 4^m  (-h_n)_{m}   \left( a_{k_1, k_1,m}(s-\bar x) \mathcal{Q}^{-m-2k_1-2}_{c,s}(x) \right. \\
\nonumber
&&\left. + \sum_{j=1}^{k_1-1} \mathbf{a}^1_{j,k_1,m} \mathcal{Q}_{c,s}^{-m-j-2-k_1}(x) (s-x_0)^{2j+1} +\sum_{j=0}^{k_1} \mathbf{b}^1_{j,k_1,m}\mathcal{Q}_{c,s}^{-m-1-k_1-j}(x) (s-x_0)^{2j}  \right)\\
&&=2^{2h_n} h_n! (-h_n)_m (s - \bar{x}) (s - x_0)^{h_n - m} \mathcal{Q}_{c,s}^{-1 - h_n}(x)\\
&&=\frac{(-1)^{h_n-\ell}}{(h_n-\ell)!} F_L^n(s,x)(s-x_0)^{h_n-\ell}.
	\end{eqnarray*}
	
The case where $\beta=h_n-m=2k_2$ is even follows by similar arguments. In this case, the coefficients, for $0\leq j<k_2$, satisfy the following equations
$$
\mathbf{a}^2_{j,k_2,m} := 2^{2j} (m+k_2+j)! \binom{k_2+j}{2j} \prod_{\ell=0}^{k_2-j-1} (h_n-m-2k_2+\ell)=0,
$$
and
$$
\mathbf{b}^2_{j,k_2,m} := 2^{2j+1} (m+k_2+j)! \binom{k_2+j}{2j+1} \prod_{\ell=0}^{k_2 - j - 1} (h_n - m - 2k_2 + \ell)=0,
$$
since $\frac{(h_n-m-k_2-1-j)!}{(h_n-m-2k_2-1)!}=\prod_{\ell=0}^{k_2-j-1} (h_n-m-2k_2+\ell)$, $\frac{(h_n - m - k_2 - j - 1)!}{(h_n - m - 2k_2 - 1)!} = \prod_{\ell=0}^{k_2 - j - 1} (h_n - m - 2k_2  + \ell)$ and $\mathbf{a}^2_{k_2,\, k_2,\, m}=2^{h_n - m} h_n!$.

\end{remark}

\section{Concluding remarks}

In a future paper, we aim to determine which function spaces can be obtained by applying the operators $D^\beta \Delta_{n+1}^m$ and $\overline{D}^\beta \Delta_{n+1}^m$ to the set of left slice hyperholomorphic functions. In other words, we seek to identify the function spaces $X(U)$ and $Y(U)$ such that
$$
D^\beta \Delta_{n+1}^m\bigl(\mathcal{SH}_L(U)\bigr) = X(U), 
\qquad 
\overline{D}^\beta \Delta_{n+1}^m\bigl(\mathcal{SH}_L(U)\bigr) = Y(U).
$$
By combining Theorems \ref{p11}, \ref{second}, and \ref{Cauinte}, we can obtain integral representations for functions in the spaces $X(U)$ and $Y(U)$. Based on these representations and the theory of the $S$-spectrum (see \cite{ColomboSabadiniStruppa2011}), we aim to establish a functional calculus for bounded operators whose components commute. These functional calculi will be more general than the polyharmonic and polyanalytic functional calculi studied in \cite{CDP25}.

\appendix
\section{}
In this Appendix, we present the main properties of the coefficients $a^1_{j,k_1,m}$, $b^1_{j,k_1,m}$, $a^2_{j,k_1,m}$, and $b^2_{j,k_1,m}$, which have been employed in the proof of Theorem \ref{p11}. The following Stifel identity will be crucial:
\begin{equation}
\label{SD}
\binom{n}{k}= \binom{n-1}{k}+\binom{n-1}{k-1}.
\end{equation}
\begin{proposition}
Let $k_2$, $m \in \mathbb{N}$. For $n$ being an odd number we set $h_n=\frac{n-1}{2}$. Then we have
\begin{enumerate}
\item 
\begin{equation}
\label{c1}
2(m+2k_2)b_{k_2-1,k_2,m}^2=2a^1_{k_2-1,k_2,m}
\end{equation}
\item For $0 \leq j \leq k_2-2$ 
\begin{equation}
\label{c2}
\left((-2j-2) a^2_{j+1,k_2,m} +2(m+k_2+j+1) b^2_{j,k_2,m} \right)=2a^1_{j,k_2,m},  
\end{equation}
\item 
\begin{equation}
\label{c3}
	2a^2_{0,k_2,m}(h_n-m-k_2) -b^2_{0,k_2,m}=2 b^1_{0,k_2,m}
\end{equation}
\item 
\begin{equation}
\label{c4}
	2 a^2_{j,k_2,m} (h_n-m-j-k_2) + 4(m+k_2+j) b^2_{j-1,k_2,m}-(2j+1) b^2_{j,k_2,m}=2b^1_{j,k_2,m}
\end{equation}
\item 
\begin{equation}
\label{c5}
		4(m+2k_2) b^2_{k_2-1,k_2,m} =2b^1_{k_2,k_2,m}.
\end{equation}
\end{enumerate}
\end{proposition}
\begin{proof}
Formula \eqref{c1} follows directly from the definitions of the coefficients $a^1_{k_2-1,k_2,m}$ and $b^2_{k_2-1,k_2,m}$.
\\Now, we prove formula \eqref{c2}. By the relation:
$$(2j+2) \binom{k_2+j}{2j+1}=(k_2-j-1)\binom{k_2+j}{2j+1}$$
and using the Stifel identity (see \eqref{SD}), we have
	\[
\begin{split}
	&\left((-2j-2) a^2_{j+1,k_2,m} +2(m+k_2+j+1) b^2_{j,k_2,m} \right)=2^{2j+2} (m+k_2+j+1)! (k_2-j-2)! \\
	& \times \left[ -(2j+2) \binom{\kd+j}{2j+2}  \binom{h_n-m-\kd-2-j}{h_n-m-2 k_2} + (k_2-j-1) \binom{h_n-m-\kd-1-j}{h_n-m-2\kd}\binom{\kd+j}{2j+1} \right]\\
	& =2^{2j+2} (k_2-j-1) ! (m+k_2+j+1)! \binom{k_2+j}{2j+1}\binom{h_n-m-k_2-j-2}{h_n-m-2k_2-1}\\
	&=2a^1_{j,k_2,m}.
\end{split}
\]
Formula \eqref{c3} follows by using the definition of the coefficients $a^2_{0,k_2,m}$ and $b^2_{0,k_2,m}$:
\[
\begin{split}
	2a^2_{0,k_2,m}(h_n-m-k_2) -b^2_{0,k_2,m} & = 2 (k_2-1)! (m+k_2)! \binom{h_n-m-k_2-1}{h_n-m-2k_2} (h_n-m-k_2) \\
	& - 2 k_2! (m+k_2)! \binom{h_n-m-k_2-1}{h_n-m-2k_2} =2 b^1_{0,k_2,m} 
\end{split}
\]
To show formula \eqref{c4} we use the relation 
$$ (h_n-m-k_2-j)\binom{h_n-m-k_2-1-j}{h_n-m-2k_2}=(k_2-j)\binom{h_n-m-k_2-j}{h_n-m-2k_2} $$
and two times the Stifel identity (see \eqref{SD}) for $\binom{k+j-1}{2j}+\binom{k+j-1}{2j-1}$ and for $\binom{h_n-m-k_2-j}{h_n-m-2k_2}-\binom{h_n-m-k_2-1-j}{h_n-m-2k_2}$:
\[
\begin{split}
	& 2 a^2_{j,k_2,m} (h_n-m-j-k_2) + 4(m+k_2+j) b^2_{j-1,k_2,m}-(2j+1) b^2_{j,k_2,m}\\
	& = 2^{2j+1} (k_2-j-1)! (m+k_2+j)!\\
	&\times \left[ \binom{\kd+j-1}{2j} \binom{h_n-m-\kd-1-j}{h_n-m-2\kd} (h_n-m-j-\kd) \right. \\
	& \left. +(k_2-j) \binom{h_n-m-\kd-j}{h_n-m-2\kd}\binom{\kd+j-1}{2j-1} -(2j+1)\binom{h_n-m-\kd-1-j}{h_n-m-2\kd}\binom{\kd+j}{2j +1} \right] \\
	&= 2^{2j+1} (k_2-j)! (m+k_2+j)!\binom{k_2+j}{2j} \left[\binom{h_n-m-j-k_2}{h_n-m-2k_2}-\binom{h_n-m-\kd-1-j}{h_n-m-2\kd} \right]\\
	& = 2^{2j+1} (m+k_2+j)! (k_2-j)! \binom{\kd+j}{2j} \binom{h_n-m-k_2-j-1}{h_n-m-2k_2-1}=2b^1_{j,k_2.m}
\end{split}
\]
Finally formula \eqref{c5} follows by the definition of $b^2_{k_2-1,k_2,m}$  and $2b^1_{k_2,k_2,m}$.
\end{proof}

\begin{proposition}
Let $k_1 \in \mathbb{N}_0$ and $m \in \mathbb{N}$. For $n$ being an odd number we set $h_n=\frac{n-1}{2}$. Then we have
\begin{enumerate}
\item For $0\leq j\leq k_1$
\begin{equation}
\label{C1}
-(2j+1) a^1_{j,k_1+1,m} + 2(m+k_1+j+2) b^1_{j,k_1+1,m}=2 a^2_{j,k_1+2,m}
\end{equation}
\item 
\begin{equation}
\label{C2}
2(m+2k_1+3) b^1_{k_1+1,k_1+1,m}=2a^2_{k_1+1,k_1+2,m}.
\end{equation}
\item For $0\leq j\leq k_1 $
\begin{equation}
\label{C3}
2(h_n-m-k_1-2-j)a^1_{j,k_1+1,m}+4(m+k_1+j+2)b^1_{j,k_1+1,m}-2(j+1) b^1_{j+1,k_1+1,m}=2b^{2}_{j,k_1+2, m}
\end{equation}
\item 
\begin{equation}
\label{C4}
4(m+2k_1 + 3) b^1_{k_1+1,k_1+1,m}=2 b^2_{k_1+1,k_1+2,m}
\end{equation}
\end{enumerate}
\end{proposition}
\begin{proof}
We start proving formula \eqref{C1}. By using the Stifel indentity (see \eqref{SD}) and the equality 
$$ (2j+1)\binom{k_1+1+j}{2j+1}=(k_1+1-j)\binom{k_1+1+j}{2j} $$
we have
\begin{align*}
	&-(2j+1) a^1_{j,k_1+1,m} + 2(m+k_1+j+2) b^1_{j,k_1+1,m}= 2^{2j+1} (m+k_1+2+j)!(k_1-j)! \\
	&\times \left[ - (2j+1) \binom{k_1+1+j}{2j+1} \binom{h_n-m-k_1-j-3}{h_n-m-2k_1-3} + \binom{h_n-m-k_1-j-2}{h_n-m-2k_1-3} \nonumber \right. \nonumber \\ 
	& \left.  \times \binom{k_1+1+j}{2j} (k_1+1-j) \right] \nonumber \\
	& = 2^{2j+1}(m+k_1+2+j)! (k_1-j+1)!\binom{k_1+j+1}{2j}\binom{h_n-m-j-3-k_1}{h_n-m-2k_1-4}\nonumber\\
	&=2 a^2_{j,k_1+2,m}
\end{align*} 
Formula \eqref{C2} follows by using the definition of $b_{k_1+1,k_1+1,m}^1$ amd $a_{k_1+1, k_1+2,m}^2$. 
\\ To show formula \eqref{C3} we use the equalities 
$$ (h_n-m-k_1-2-j)\binom{h_n-m-k_1-3-j}{h_n-m-2k_1-3}=(k-j+1)\binom{h_n-m-k_1-j-2}{h_n-m-2k_1-3}$$ 
and
$$ (2j+2)\binom{k_1+j+2}{2j+2} = (k_1-j+1)\binom{k_1+j+2}{2j+1} $$
and two times Stifel identity (see \eqref{SD}):
\begin{align*}
	& 2(h_n-m-k_1-2-j)a^1_{j,k_1+1,m}+4(m+k_1+j+2)b^1_{j,k_1+1,m}-2(j+1) b^1_{j+1,k_1+1,m} \nonumber\\
	&= 2^{2j+2}(m+k_1+2+j)!  (k_1-j)!  \left[ (h_n-m-k_1-2-j)\binom{k_1+1+j}{2j+1}  \binom{h_n-m-k_1-3-j}{h_n-m-2k_1-3} \right. \nonumber \\
	&  \left. +(k-j+1) \binom{h_n-m-k_1-j-2}{h_n-m-2k_1-3}  \binom{k_1+1+j}{2j} -(2j+2)  \binom{h_n-m-k_1-j-3}{h_n-m-2k_1-3}  \binom{k_1+j+2}{2j +2} \right]     \nonumber \\ 
	&= 2^{2j+2} (m+k_1+2+j)! (k-j+1)! \binom{k_1+2+j}{2j+1} \left[\binom{h_n-m-k_1-2-j}{h_n-m-2k_1-3}- \binom{h_n-m-k_1-j-3}{h_n-m-2k_1-3}\right]\\
	&  = 2^{2j+2} (m+k_1+2+j)! (k_1-j+1)! \binom{k_1+2+j}{2j+1}\binom{h_n-m-k_1-3}{h_n-m-2k_1-4} = 2b^{2}_{j,k_1+2, m} 
\end{align*}
Finally, formula \eqref{C4} follows by the definitions of the $b^{2}_{k_2-1,k_2,m}$ and $b^{1}_{k_2,k_2,m}$.
\end{proof}

	\section*{Declarations and statements}
	
	{\bf Data availability}. The research in this paper does not imply use of data.
	
	{\bf Conflict of interest}. The authors declare that there is no conflict of interest.
	
	{\bf Acknowledgments}.
 A. De Martino and S. Pinton are supported by MUR grant Dipartimento di Eccellenza 2023-2027.

\end{document}